\documentclass[onefignum,onetabnum]{siamart190516}

\usepackage{iftex}
\ifPDFTeX
  \usepackage[utf8]{inputenc} 
  \usepackage[T1]{fontenc} 
  \usepackage{lmodern} 
\else
  \usepackage{fontspec} 
\fi

\usepackage[english]{babel} 
\usepackage{hyperref}

 \usepackage{cite}                            
\usepackage{subfigure}

\usepackage{tikz}
\usepackage{pgfplots}
\pgfplotsset{compat=newest,compat/show suggested version=false}
\usetikzlibrary{pgfplots.groupplots}

\def\derpar#1#2{\frac{\partial#1}{\partial#2}}
\newcommand{\ens}[1]{\mathbb{#1}}

\def\RR{\mathbb R}
\def\ve{\varepsilon}
\def\R{\mathbb R}

\def\P{\mathcal P}
\def\proj{\mathcal P}

\def\B{\mathcal B}
\def\PP{\mathbb P}
\def\poly{\mathbb P}
\def\f{\hat f}
\def\LL{\mathcal L}
\def\supp{{\rm Supp}}
\def\Ball{ {\cal B}}
\DeclareMathOperator{\sinc}{Sinc}
\DeclareMathOperator{\spann}{span}
\def\vve{\delta}

\def\be{\begin{equation}}
\def\ee{\end{equation}}
\def\bea{\begin{eqnarray}}
\def\eea{\end{eqnarray}}
\def\beas{\begin{eqnarray*}}
\def\eeas{\end{eqnarray*}}

\usepackage{lipsum}
\usepackage{amsfonts}
\usepackage{graphicx}
\usepackage{epstopdf}
\usepackage{algorithmic}
\ifpdf
  \DeclareGraphicsExtensions{.eps,.pdf,.png,.jpg}
\else
  \DeclareGraphicsExtensions{.eps}
\fi

\newsiamremark{remark}{Remark}
\newsiamremark{hypothesis}{Hypothesis}
\crefname{hypothesis}{Hypothesis}{Hypotheses}
\newsiamthm{claim}{Claim}

\usepackage{amsopn}

\title{Moment preserving Fourier-Galerkin spectral methods and application to the Boltzmann equation}
\author{Lorenzo Pareschi\thanks{Department of Mathematics \& Computer Science, University of Ferrara, Via Machiavelli 30, Ferrara, 44121, Italy (\email{lorenzo.pareschi@unife.it}).} \and Thomas Rey\thanks{Univ. Lille, CNRS, UMR 8524, Inria – Laboratoire Paul Painlevé F-59000 Lille, France (\email{thomas.rey@univ-lille.fr})}}

\begin{document}

\maketitle

\begin{abstract} Spectral methods, thanks to the high accuracy and the possibility of using fast algorithms, represent an effective way to approximate collisional kinetic equations in kinetic theory. On the other hand, the loss of some local invariants can lead to the wrong long time behavior of the numerical solution. We introduce in this paper a novel Fourier-Galerkin spectral method that improves the classical spectral method by making it conservative on the moments of the approximated distribution, without sacrificing its spectral accuracy or the possibility of using fast algorithms. The method is derived directly using a constrained best approximation in the space of trigonometric polynomials and can be applied to a wide class of problems where preservation of moments is essential. 
We then apply the new spectral method to the evaluation of the Boltzmann collision term, and prove spectral consistency and stability of the resulting Fourier-Galerkin approximation scheme. Various numerical experiments illustrate the theoretical findings.\\[.5em]
    \textsc{Keywords:} Boltzmann equation, Fourier-Galerkin spectral method, conservative methods, spectral accuracy, stability, Maxwellian equilibrium.\\[.5em]
    \textsc{2010 Mathematics Subject Classification:} 76P05, 
    65N35, 
    82C40. 
\end{abstract}

\tableofcontents

\section{Introduction}
Spectral methods for collisional kinetic equations have a long history. They have been originally inspired by the pioneering works on the Fourier transformed Boltzmann equation for Maxwell molecules by A. Bobylev \cite{Bobylev:75} which later inspired the first methods based on finite difference discretizations of the Fourier transform \cite{BR97, BR99, BoRj:98}. In such approaches the main purpose was to exploit the simplified form of the equation in Fourier space rather than the construction of a method providing spectral accuracy. The first Fourier-Galerkin type spectral method was introduced in the same years in \cite{PP96, PR00, PR00b}. Thanks to the new formalism, it was possible to prove spectral accuracy and consistency of the method. In particular, the method lent itself to be generalized to other collisional kinetic equations such as the Landau equation, the inelastic Boltzmann equation for granular gases and the quantum Boltzmann equation \cite{DLPY2015, PARESCHI2000, filbet:2005, PareschiToscaniVillani:2003, FHJ2012}. 

While in the case of the Landau equation the development of fast algorithms with spectral accuracy was achieved immediately due to the convolutional structure of the collision operator, in the case of the Boltzmann equation it represented a major breakthrough achieved later in \cite{MP06, FiMoPa:2006}. Subsequent developments that have extended the construction of fast algorithms to inelastic collisions, quantum interactions and general collisional kernels have been obtained in \cite{HU2019, gamba2017fast, wu2013deterministic,HY2012}. Due to these advances, spectral methods have gained quite a bit of popularity in numerical simulations of the space non homogeneous Boltzmann equation \cite{gamba:2010,FILBET2003,LI2014} and today are successfully used in realistic multidimensional applications   \cite{wu2013deterministic,WuZhangHynReeseZhang:2015,jaiswal2019fast,DIMARCO2018}. For a more comprehensive introduction to this class of methods and further references we refer the interested reader to the survey in \cite{DimarcoPareschi15}.

The main advantages of a Fourier-Galerkin type approach are the spectral accuracy for smooth solutions and the possibility to use fast algorithms that mitigate the curse of dimensionality. On the other hand, they typically lead to the loss of most physical properties of the Boltzmann equation, namely positivity, conservations, entropy dissipation and, as a consequence, long time behavior. The construction of deterministic numerical methods that can preserve the collisional invariants has always been a challenge in the approximation of the Boltzmann equation and related kinetic problems \cite{DimarcoPareschi15, DPR04, BG03}. In fact, a major problem associated with deterministic
methods is that the velocity space is approximated by a finite region. On the other hand, even starting from a compactly supported function in velocity, by the action of the collision term the solution becomes immediately positive in the whole velocity space. In particular, the local Maxwellian equilibrium states are characterized by exponential functions defined on the whole velocity space. Another line of research is based on the use of different orthogonal polynomials that do not require truncation of the velocity space, we refer the reader to \cite{WC19,GR18,cai2018entropic,HU21} and the references therein for more details. 
 
Thus, at the numerical level some non physical conditions have to be
imposed to keep the support of the function in velocity uniformly
bounded. This can be done by neglecting collisions which will spread the support of the solution outside the finite region, as in discrete velocity models \cite{HePa:DVM:02,RoSc:quad:94}, or by periodizing the function and the collision term, as in spectral methods \cite{PP96, PR00} or in finite difference schemes based on the Fourier transform \cite{BR97, BR99, Gamba:2009}. In the former case, however, the symmetries of the Boltzmann collision operator that underlie the development of fast solvers are destroyed and a periodized formulation is therefore a necessary condition to derive computationally efficient algorithms also for discrete velocity models \cite{MPR:13}.  

Spectral methods preserving some physical properties have been introduced by various authors using smoothing or renormalization techniques at different levels \cite{PP96, BR99, BoRj:98,Gamba:2009,PR00b, cai2018entropic}. However, these approaches typically may lead to the loss of spectral accuracy of the resulting approximation scheme. A way to overcome some of these drawbacks by keeping spectral accuracy has been proposed recently \cite{FilbetPareschiRey:2015, PareschiRey2020}, where Fourier-Galerkin steady state preserving spectral methods have been constructed. 
However, this class of methods does not exactly preserve moments and the approach can only be generalized to kinetic equations where the equilibrium state is known.
On the other hand, it has been observed in \cite{wu2013deterministic}, through numerical experiments, that the Lagrangian multiplier approach introduced in \cite{Gamba:2009} when applied to the Fourier-Galerkin approximation is capable of maintaining spectral accuracy. We refer also to \cite{alonso2018convergence} for recent results on the convergence properties of the method in \cite{Gamba:2009}.

Motivated by the previous discussion, we consider in this paper a novel Fourier-Galerkin fast spectral methods that improves the classical fast spectral scheme by making it conservative on the moments of the approximated distribution, while maintaining the theoretical properties of spectral accuracy, consistency and stability. The method is derived directly starting from a constrained best approximation in the space of trigonometric polynomials and can be applied in a standard Fourier-Galerkin setting to a wide class of kinetic equations where preservation of moments is essential \cite{Vill:hand}. Due to its relevance in applications, even if the method is introduced in its generality, we discuss in details the application to the challenging case of the Boltzmann equation and show that, thanks to the new formalism, previous results on spectral consistency and stability can be extended to the present case.

The rest of the manuscript is organized as follows. In the next Section we recall some essential facts about the Boltzmann equation and the corresponding periodized space homogeneous problem. Section 3 is then devoted to the introduction of the conservative finite Fourier series as the constrained best approximation in the space of trigonometric polynomials.   
This permits to extend to the conservative case the classical result of spectral convergence for smooth solutions. In Section 4 we apply the new method to the Boltzmann equation in a Fourier-Galerkin setting and prove spectral consistency and stability of the resulting scheme under a smallness assumption on the loss of conservation of the periodized collision term. To simplify the presentation and have a more self-contained treatment, the fast spectral method is summarized in Appendix A together with the details of other spectral schemes used in the numerical section. The subsequent Section illustrates through several numerical examples the performance of the new method in comparison with the classical fast spectral method \cite{MP06} and the equilibrium preserving fast spectral method \cite{FilbetPareschiRey:2015}. Finally we end the manuscript with some concluding remarks and future developments. 

\section{The Boltzmann equation}

\subsection{Basic properties of the equation}
The Boltzmann equation describes the behavior of a dilute gas of
particles when the only interactions taken into account are binary
elastic collisions \cite{CIP:94,Vill:hand}. It reads for $x \in \R^{d}$, $v \in \R^{d}$, $d \leq 3$
  \begin{equation}
  \derpar{f}{t} + v \cdot \nabla_x f = Q(f,f)
  \label{eq:full}
  \end{equation}
where $f(t,x,v)$ is the time-dependent particle distribution
function in the phase space. The Boltzmann collision operator $Q$
is a quadratic operator local in $(t,x)$. The time and position
acts only as parameters in $Q$ and therefore will be omitted in
its description
  \begin{equation}\label{eq:mp:Q}
  Q (f,f)(v) = \int_{\R^{d} \times\ens{S}^{d-1}} B(|v-v_*|,\cos \theta) \,
  \left( f'_* f' - f_* f \right) \, dv_* \, d\sigma.
  \end{equation}
In~\eqref{eq:mp:Q} we used the shorthand $f = f(v)$, $f_* = f(v_*)$,
$f ^{'} = f(v')$, $f_* ^{'} = f(v_* ^{'})$. The velocities of the
colliding pairs $(v,v_*)$ and $(v',v'_*)$ can be parametrized as
  \begin{equation*}
  v' = \frac{v+v_*}{2} + \frac{|v-v_*|}{2} \sigma, \qquad
  v'_* = \frac{v+v^*}{2} - \frac{|v-v_*|}{2} \sigma.
  \end{equation*}
The collision kernel $B$ is a non-negative function which by
physical arguments of invariance only depends on $|v-v_*|$ and
$\cos \theta = {\hat g} \cdot \sigma$ (where ${\hat g} =
(v-v_*)/|v-v_*|$).
It characterizes the
details of the binary interactions, and has the form
\begin{equation}
\label{defVHSKernel}
B(|v-v_*|,\cos\theta)=|v-v_*|\sigma(|v-v_*|,\cos\theta)
\end{equation}
where the {scattering cross-section} $\sigma$, in the case of inverse $k$-th power forces between particles, can be written as 
\[ 
\sigma(\vert v - v_{\ast} \vert,
\cos\theta) = b_{\alpha}(\cos\theta) \vert v - v_{\ast}
\vert^{\alpha-1}, 
\]
with $\alpha=(k-5)/(k-1)$. 
The special situation {$k=5$} gives the so-called {Maxwell pseudo-molecules model} with  
\[ 
B(|v-v_*|,\cos\theta)=
b_{0}(\cos\theta). 
\]
For the Maxwell case the collision kernel is independent of the relative velocity. For numerical purposes, a
widely used model is the {variable hard sphere} (VHS) 
model introduced by Bird~\cite{Bird:1994}. The model corresponds to {$b_{\alpha}(\cos\theta)=C_\alpha$,}
where {$C_\alpha$} is a positive constant, and hence
\[
\sigma(\vert v - v_{\ast} \vert, \cos\theta) =
C_{\alpha} \vert v - v_{\ast} \vert^{\alpha-1}. 
\]
In the numerical test Section we will consider the Maxwell molecules case when dealing with a velocity space of dimension $d=2$.

Boltzmann's collision operator has the fundamental properties of
conserving mass, momentum and energy
  \begin{equation}
  \int_{\R^d} Q(f,f)\,\phi(v) \,dv = 0, \quad
  \phi(v)=1,v_1,\ldots,v_d,|v|^2,
  \label{eq:cons}
 \end{equation}
and satisfies the well-known Boltzmann's $H$ theorem
  \begin{equation*} 
  \frac{d}{dt} \int_{\R^d} f \log f \, dv = \int_{\R^d} Q(f,f) \log(f) \, dv \leq 0.
  \end{equation*}
Boltzmann's $H$ theorem implies that any equilibrium
distribution function has the form of a locally Maxwellian distribution
  \begin{equation}
  M(\rho,u,T)(v)=\frac{\rho}{(2\pi T)^{d/2}}
  \exp \left\{ - \frac{\vert u - v \vert^2} {2T} \right\},
  \label{eq:maxw}
  \end{equation}
where $\rho,\,u,\,T$ are the density, mean velocity
and temperature of the gas
  \begin{equation*}
  \rho = \int_{\R^d}f(v) \, dv, \quad
  u = \frac{1}{\rho}\int_{\R^d} v f(v) \, dv, \quad
  T = \frac{1}{d\rho} \int_{\R^d}\vert u - v \vert^2f(v) \, dv.
  \end{equation*}
For further details on the physical background and derivation of
the Boltzmann equation we refer to~\cite{CIP:94,Vill:hand}.

In the sequel we will restrict out attention to the space homogeneous setting where $f=f(v,t)$ satisfies
\be
\left \{
\begin{aligned}
\frac{\partial f}{\partial t } &= Q(f,f)\\
f(v,0)&=f_0(v),\quad v\in \RR^d.
\label{eq:homof}
\end{aligned} \right.
\ee
In fact, the numerical solution of \eqref{eq:homof} contains all the major difficulties related to the Boltzmann equation \eqref{eq:full}, and can be easily extended to the full inhomogeneous case by splitting algorithms or IMEX methods (see \cite{DimarcoPareschi15}).

\subsection{Reduction to a bounded domain and periodization}
\label{ssec:bd}
As shown in \cite{PP96}, we have the following 
\begin{proposition} Let the distribution function $f$ be compactly supported on the ball $\Ball_0(R)$ of
radius $R$ centered in the origin, then  
\begin{equation*}
\supp (Q(f,f)(v)) \subset \Ball_0({\sqrt 2}R).
\end{equation*} 
\end{proposition}
In order to write a spectral approximation which avoid superposition of periods, it is then sufficient that the distribution function $f(v)$ is restricted on the cube $[-T,T]^{d}$ with $T \geq (2+{\sqrt 2})R$. Successively, one should assume $f(v)=0$ on $[-T,T]^{d} \setminus \Ball_0(R)$ and extend $f(v)$ to a periodic function on the set $[-T,T]^{d}$. Let observe that the lower bound for $T$ can be improved. For instance, the choice $T=(3+{\sqrt 2})R/2$ guarantees the absence of intersection between periods where $f$ is different from zero \cite{PR00}. However, since in practice the support of $f$ increases with time, we can just minimize the errors due to aliasing \cite{canuto:88} with spectral accuracy. 
To further simplify the notation, let us take $T=\pi$ and hence $R=\lambda\pi$ with $\lambda = 2/(3+\sqrt{2})$ in the sequel.

We shall consider in the rest of the paper the following periodized, space homogeneous problem \cite{PR00, FM11} 
\be
\left \{
\begin{aligned}
\frac{\partial f}{\partial t } &= Q^R(f,f)\\
f(v,0)&=f_0(v),\quad v\in [-\pi,\pi]^d,
\label{eq:homo}
\end{aligned} \right.
\ee
where $Q^R$ is the truncated and periodized collision term 
\begin{equation}
	\label{eq:truncatedQ}
	Q^R(f,f)(v) = \int_{\Ball_0(2R) \times\ens{S}^{d-1}} B(|v-v_*|,\cos \theta) \,
  \left( f'_* f' - f_* f \right) dv_* \, d\sigma.
\end{equation}

Note that,	because of the reduction to a bounded domain by periodization, the collisional invariants of the original problem are lost, except for mass conservation. As a consequence the local equilibria $m_\infty$ of problem \eqref{eq:homo}-\eqref{eq:truncatedQ} are the (piecewise) constant functions
	\begin{equation}
		\label{def:constEquilibrium}
		m_\infty(v) := \frac \rho{2\pi^d}, \quad \forall\,\, v \in [-\pi,\pi]^{d}.
	\end{equation}
Concerning the loss of conservations of the periodized collision term in $[-\pi,\pi]^d$ we have the following result:
\begin{proposition}
Assuming the solution to problem \eqref{eq:homo} satisfies for $\vve \ll 1$
\begin{equation*}
f(v,t) \leq \vve,\quad v\in[-\pi,\pi]^d \setminus \Ball_0(R)  
\end{equation*}
with $R=\lambda\pi$. Then, we have the bound for  $\Phi(v)=(1,v_1,\ldots,v_d,|v|^2)^T$
\begin{equation}
\left\|\int_{[-\pi,\pi]^d} Q^R(f,f)\, \Phi(v)\,dv \right\|_2 \leq C\vve,
\label{eq:epsin}
\end{equation}
where $C=C(f,R)$ and $\|\cdot\|_2$ denotes the euclidean norm of the vector.
\label{thm:suppQR}
\end{proposition}
\begin{proof}
From the assumption on $f(v,t)$ we can write
\[
f(v,t)=f^c(v,t)+f^\vve(v,t)
\]
where $f^c(v,t)$ is compactly supported in $\Ball_0(R)$ and $f^\vve(v,t)=0$ in $\Ball_0(R)$ with $f^\vve(v,t) \leq \vve$ in $[-\pi,\pi]^d\setminus \Ball_0(R)$.

Since
\[
Q^R(f,f) = Q^R(f^c,f^c)+Q^R(f^\vve,f^\vve)+Q^R(f^c,f^\vve)+Q^R(f^\vve,f^c),
\]
estimate \eqref{eq:epsin} follows from the weak form
\begin{equation*}
\begin{split}
\int_{[-\pi,\pi]^d} Q^R(f,f)\phi(v)\,dv =&\\
 \int_{[-\pi,\pi]^{d}}&\int_{\Ball_0(2R)\times\mathbb{S}^{d-1}} B(|v-v_*|,\cos\theta) f f_*(\phi(v')-\phi(v))\,dv_*\,d\sigma\,dv, 
 \end{split}
\end{equation*}
and the conservation property
\[
\int_{[-\pi,\pi]^d} \!\!\!Q^R(f^c,f^c)\,\Phi(v)\,dv = 0.
\]
\end{proof}

{For the Boltzmann equation \eqref{eq:homof}, it is well known (and precise estimates are available in \cite{gamba2009upper}) that solutions $f$ have uniform in time maxwellian upper bounds: There exists constants $C \geq 1$ and $0 < \mu \leq 1$, depending on the collision kernel $B$ and the initial datum $f_0$ such that
\[
	0 \leq f(t,v) \leq C \frac \rho{\left (2 \pi T\right )^{d/2}} e^{-\mu\frac{|v - u|^2}{2T}} = \frac C{\mu^{d/2}} M(\rho, u, T)(v)\quad \forall\, v \in \RR^d, t >0. 
\]
This, combined with the previous proposition would provide a bound on the loss of moments due to truncation.
Note, however, that a similar estimate for the periodized problem
\eqref{eq:homo} is not available since asymptotically the solution
converges to the steady state \eqref{def:constEquilibrium}.
%

Due to the physical relevance of the collisional invariants, which are the basis of the long time behavior \eqref{eq:maxw}, usually the truncated collision operator is modified to maintain the original conservation properties \eqref{eq:cons}. For example, by modifying the collision kernel in order to neglect collisions originating velocities outside the bounded domain as in discrete velocity models \cite{HePa:DVM:02}. This, however, destroying the convolutional structure of \eqref{eq:truncatedQ}, leads to computationally inefficient quadratures that cannot take advantage of fast solvers \cite{MPR:13, MP06}. In our setting, as we will illustrate, conservation are recovered as a constrained best approximation in the Fourier space which permits a standard implementation of fast algorithms.

%

\section{Conservative approximations by trigonometric polynomials}
\label{sec:specContraint}
In order to get conservative spectral methods, we construct a conservative projection of the periodized solution on the space of trigonometric polynomials following the constrained formulation approach introduced in \cite{BR99} and \cite{Gamba:2009} for finite difference discretizations of the Fourier transformed Boltzmann equation.
In particular, we will give an explicit formulation of the trigonometric polynomial of best approximation in the least square sense, constrained by preservation of moments, and show that it preserves spectral accuracy for smooth solutions like the classical trigonometric approximation by finite Fourier series.

\subsection{The finite Fourier series}
	Let us first set up the mathematical framework of our analysis. For the sake of the reader convenience, we shall restrain ourselves to the domain $[-\pi,\pi]^d$.  Given a function $f(v)\in L^2_p([-\pi,\pi]^d)$, $d\geq 1$, we set as before its Fourier series representation as
\begin{equation}
f(v) = \sum_{k=-\infty}^\infty \f_k e^{i k \cdot v}, \qquad \f_k = \frac{1}{(2\pi)^d}\int_{[-\pi,\pi]^d} f(v)
e^{-i k \cdot v }\,dv,
\label{eq:FS}
\end{equation}
where we use just
one vector index $k=(k_1,\ldots,k_d)$ to denote the $d$-dimensional sums over the indexes $k_j$, $j=1,\ldots,d$.

We define the space $\PP^N$ of trigonometric polynomials of degree $N$ in
$v$ as
\[
\poly^N = \spann\left\{e^{ik\cdot v}\,|\, -N \leq k_j \leq N,\, j=1,\ldots,d
\right\}.
\]
Let $\proj_N : L^2_p([-\pi,\pi]^d) \rightarrow \poly^N$ be the
orthogonal projection upon $\poly^N$ in the inner product of
$L^2_p([-\pi,\pi]^d)$ 
\[
\langle f-\proj_N f,\phi\rangle=0,\qquad \forall\,\, \phi\,\in\,\poly^N.
\]
With these definitions $\proj_N f=f_N$, where $f_N$ is the finite
Fourier series of $f$ given by 
\begin{equation}
f_N(v) = \sum_{k=-N}^N \f_k e^{i k \cdot v},
\label{eq:truncatedFourierSum}
\end{equation}
Since the operator $\proj_N$ is
self-adjoint \cite{canuto:88} the following property hold
\begin{equation*}
\langle\proj_N f,\varphi\rangle=\langle f,\proj_N\varphi\rangle=\langle\proj_N
f,\proj_N\varphi\rangle\quad\forall\,\,f,\,\varphi\in L^2_p([-\pi,\pi]^d).
\end{equation*}
We define the $L^2_p$-norm by
\[
\|f\|^2_{L^2_p} = \langle f, f\rangle.
\]
By the Parseval's identity we have
\begin{equation*}
\| f \|^2_{L^2_p} = (2\pi)^d \sum_{k=-\infty}^\infty |\hat f_k|^2
,\qquad \| f_N \|^2_{L^2_p} = (2\pi)^d \sum_{k=-N}^N |\hat f_k|^2.
\end{equation*}
An important feature of the orthogonal projection on $\PP^N$ represented by the truncated Fourier series \eqref{eq:truncatedFourierSum} is related to its spectral convergence properties for smooth solutions which are summarized in the following theorem.
\begin{theorem}
\label{thm:ClassicalFourierSpectralAccuracy}
If $f\in H_p^r([-\pi,\pi]^d)$, where $r \geq 0$ is an integer and $H^r_p([-\pi,\pi]^d)$ is the subspace of the Sobolev space $H^r([-\pi,\pi]^d)$ which consists of periodic functions, we have
\be
\|f-f_N\|_{H^r_p} \leq \frac{C}{N^r} \|f\|_{H_p^r}.
\label{eq:ses}
\ee
\end{theorem}
Finally, we recall some approximation properties of the projection
operator $\proj_N$, in particular those concerning approximation of the macroscopic quantities. Let us remark that, in general, when
we approximate a function by a partial sum of its Fourier series, except for the moment of order zero all other moments are not preserved by the projection.

The results are summarized in the following proposition \cite{PR00}

\begin{proposition}
Let $f \in L^2_p([-\pi,\pi]^d)$ and let us
define
\begin{equation*}
U=(\rho,\rho u, \rho e)^T =\langle f,\Phi\rangle,
\end{equation*}
where $\Phi=(1,v_1,\ldots,v_d,|v|^2)^T\in\R^{d+2}$. 
\begin{description}
\item{i)}
The moments of $f_N$ are given by
\be
U_N=(\rho_N,\rho u_N, \rho e_N)^T =\langle f_N,\Phi\rangle =
 (2\pi)^d\sum_{k=-N}^{N} \f_k \hat \Phi_k,
 \label{eq:mfn}
\ee
where $\hat \Phi_k = (\delta_{k0}, \hat v_{k}, \hat{(v^2)}_{k})^T\in\R^{d+2}$, $\delta_{k0}$ is the
Kronecker delta, and $\hat v_k$ and $\hat{(v^2)}_k$ are the Fourier
coefficients of $v$ and $v^2$ characterized by
$\hat{(v_j)}_{0}=0$, $j=1,\ldots,d$ and $\hat{(v^2)}_{0}=\pi^2$ whereas for
$k \neq 0$ we get
\begin{equation}
\hat{(v_j)}_{k} = -i \prod_{\substack{l=1 \\ l\neq j}}^d \delta_{k_l
0}\frac{(-1)^{k_j}}{k_j},\,\, \quad j=1,\ldots,d,\quad
\hat{(v^2)}_{k} = 2 \sum_{j=1}^d \prod_{\substack{l=1 \\ l\neq j}}^d
\delta_{k_l 0}
\frac{(-1)^{k_j}}{k_j^2}.
\label{eq:CFE}
\end{equation}

\item{ii)}
The following relations hold
\[
\rho = \rho_N,\quad |\rho u - \rho u_N | \leq \frac{C_1}{N^{1/2}}\|f\|_{L^2_p},\quad
|\rho e - \rho e_N | \leq \frac{C_2}{N^{3/2}}\|f\|_{L^2_p}.
\]

\item {iii)} If $f\in H_p^r([-\pi,\pi]^d)$, where $r \geq 0$ is an integer, for each $\phi\in L^2_p([-\pi,\pi]^d)$ we have
\[
|\langle f,\phi \rangle - \langle f_N,\phi \rangle| \leq \|\phi\|_{L^2_p} \|f-f_N\|_{H_p^r} \leq \frac{C}{N^r} \|\phi\|_{L^2_p}  \|f\|_{H_p^r}.
\label{eq:mes}
\]
\end{description}
\label{pr:2}
\end{proposition}
The last inequality shows that the projection error on the moments decay faster than algebraically when the
solution is infinitely smooth.

\subsection{Constrained best approximations}

We want to define a different projection operator on the space of trigonometric polynomials, $\P^c_N: L^2_p([-\pi,\pi]^d)\to \PP^N$ such that it satisfies 
\[ \langle \P^c_N f,\Phi \rangle  =
\langle f,\Phi \rangle,
\]
but preserving the convergence properties of the finite Fourier series.

To this aim, we recall that, given a function $f\in L^2_p([-\pi,\pi]^d)$, the truncated Fourier series \eqref{eq:truncatedFourierSum} represents the trigonometric polynomial of best approximation in the least squares sense, more precisely $f_N=\P_N f$ is the solution of the following minimization problem
\[
f_N = {\rm argmin} \left\{\| g_N - f \|^2_{L^2_p}\,\, :\,\,g_N\in\PP^N\,\right\}.
\]
Thus, it is natural to consider the following constrained best approximation problem in the space of trigonometric polynomials
\be
\label{def:constrainedApprox}
f^c_N = {\rm argmin} \left\{\| g_N - f \|^2_{L^2_p}\,\, :\,\,g_N\in\PP^N,\,\, \langle g_N,\Phi \rangle  =
\langle f,\Phi \rangle\right\}.
\ee 
Now, since $g_N\in\PP^N$ we can represent it in the form
\[
g_N=\sum_{k=-N}^N \hat g_k e^{ik\cdot v}
\]
and then by Parseval's identity
\[
\| g_N - f \|^2_{L^2_p} = (2\pi)^d \sum_{k=-\infty}^\infty |\hat g_k-\hat f_k|^2,
\]
where we assumed $\hat g_k=0$, $|k_j| > N$, $j=1,\ldots,d$.

Note that, since conservation of moments is built in $g_N$, one necessarily needs that
\be
\langle g_N,\Phi \rangle = (2\pi)^d \sum_{k=-N}^N \hat g_k \hat \Phi_k =\langle f,\Phi \rangle=U.
\label{eq:mom}
\ee

Let us now solve the minimization problem \eqref{def:constrainedApprox} using the Lagrange multiplier method.
Let $\lambda\in\R^{d+2}$ be the vector of Lagrange multipliers, we consider the objective function 
\[
\LL(\hat g,\lambda) = (2\pi)^d \sum_{k=-\infty}^\infty |\hat g_k-\hat f_k|^2 + \lambda^T\left((2\pi)^d \sum_{k=-N}^N \hat g_k \hat \Phi_k -U\right),
\]
with $\hat g\in\R^{2N+1}$ the vector of coefficients $\hat g_k$, $k=-N,\ldots,N$.  
Stationary points are found by imposing
\begin{eqnarray*}
\frac{\partial \LL(\hat g, \lambda)}{\partial \hat g_k} = 0,\quad k=-N,\ldots,N;\qquad
\frac{\partial \LL(\hat g, \lambda)}{\partial \lambda_j} = 0,\quad j=1,\ldots,d+2.
\end{eqnarray*}
From the first condition, one gets
\begin{equation}
	\label{eq:cond1Lagrange}
	2({\hat g}_k-{\hat f}_k)+ \lambda^T \hat\Phi_k=0,
\end{equation}
whereas the second condition corresponds to \eqref{eq:mom}.

Multiplying the above equation by $\hat \Phi_k$ and summing up over $k$ we can write
\[
2 \sum_{k=-N}^N ({\hat g}_k-{\hat f}_k)\hat \Phi_k+   \sum_{k=-N}^N \hat \Phi_k\hat\Phi^T_k\lambda=0
\]
and, using  \eqref{eq:mfn}, \eqref{eq:mom} and the fact that $\hat \Phi_k\hat\Phi^T_k$
are symmetric and positive definite matrices of size $d+2$ one obtains
\be
	\label{eq:Lambda}
	\lambda = -\frac{2}{(2\pi)^{d}}\left(\sum_{k=-N}^N \hat \Phi_k\hat\Phi^T_k\right)^{-1}(U-U_N).
\ee 
Now, rewriting the first condition \eqref{eq:cond1Lagrange} as
\[
	{\hat g}_k = {\hat f}_k - \frac{1}{2}  \hat\Phi_k^T \lambda,
\]
and plugging the expression \eqref{eq:Lambda} of $\lambda$ in it, one obtains that the minimum is achieved for $\hat g_k = \hat f_k^c$, given by the following definition.

\begin{definition}
One can define a \emph{conservative projection} $\P_N^c f = f_N^c$ in $\PP^N$, where $f_N^c$ is given by
\be
f_N^c = \sum_{k=-N}^N \hat f^c_k \, e^{ik\cdot v}.
\label{eq:cp}
\ee
with
\be
{\hat f^c_k} ={\hat f_k} + \hat C^T_k (U-U_N),\qquad \hat C^T_k = \frac1{(2\pi)^d}\hat \Phi^T_k \left(\sum_{h=-N}^N \hat \Phi_h\hat\Phi^T_h\right)^{-1}.
\label{eq:minim1}
\ee
\label{def:1}
\end{definition}


The following result states the spectral accuracy of the conservative best approximation in least square sense \eqref{eq:cp}, and generalizes Theorem \ref{thm:ClassicalFourierSpectralAccuracy}.
\begin{theorem}
\label{thm:SpectralAccuracyfNc}
If $f\in H_p^r([-\pi,\pi]^d)$, where $r \geq 0$ is an integer, we have
\be
\| f - f^c_N\|_{H^r_p} \leq \frac{C_\Phi}{N^r} \|f\|_{H_p^r} 
\label{eq:sec}
\ee
where the constant $C_\Phi$ depends on the spectral radius of the matrix $\langle\Phi,\Phi^T\rangle$, and on $\|\Phi\|^2_{2,L^2_p}=\sum_{j=1}^{d+2}\|\Phi_j\|^2_{L^2_p}$ 
where $\Phi_j$, $j=1,\ldots,d+2$ are the components of the vector $\Phi$.
\end{theorem}
\begin{proof}
We can split
\[
\| f - f^c_N\|^2_{L^2_p} \leq \| f - f_N\|^2_{L^2_p}+\| f_N - f^c_N\|^2_{L^2_p}.
\]
The first term is bounded by the spectral estimate of truncated Fourier series \eqref{eq:ses}, whereas
for the second term by Parseval's identity we have 
\[
\| f^c_N - f_N \|^2_{L^2_p} = (2\pi)^d \sum_{k=-N}^N |\hat f^c_k-\hat f_k|^2.
\]
Now, using the definition \eqref{eq:minim1}
we get
\begin{eqnarray*}
\sum_{k=-N}^N |\hat f^c_k-\hat f_k|^2 &=& \sum_{k=-N}^N |\hat C^T_k (U-U_N)|^2 \leq \|U-U_N\|_2^2 \sum_{k=-N}^N \|\hat C^T_k\|_2^2.
\end{eqnarray*}

From \eqref{eq:mes} spectral accuracy of moments implies
\[
\|U-U_N\|_2 \leq \frac{C}{N^r} \|f\|_{H^r_p} \left(\sum_{j=1}^{d+2}\|\Phi_j\|^2_{L^2_p}\right)^{1/2} = \frac{C}{N^r}\|f\|_{H^r_p}\|\Phi\|_{2,L^2_p}
\]
where for $\Phi=(\Phi_1,\ldots,\Phi_{2+d})^T$ we defined 
$\|\Phi\|^2_{2,L^2_p}=\sum_{j=1}^{d+2}\|\Phi_j\|^2_{L^2_p}$.

Finally, since $\hat C^T_k = (2\pi)^{-d}\hat \Phi^T_k \left(\sum_{h=-N}^N \hat \Phi_h\hat\Phi^T_h\right)^{-1}$, one has
\begin{eqnarray*}
(2\pi)^d \sum_{k=-N}^N \|\hat C^T_k\|_2^2 && \leq \sum_{k=-N}^N \|\hat \Phi_k\|_2^2 \left\|\left(\sum_{h=-N}^N \hat \Phi_h\hat\Phi^T_h\right)^{-1}\right\|^2_2\\
&& = \frac{\|\Phi\|^2_{2,L_p^2}}{\left\|\sum_{h=-N}^N \hat \Phi_h\hat\Phi^T_h\right\|^2_{2}} = \frac{\|\Phi\|^2_{2,L_p^2}}{\left\|\langle\Phi,\Phi^T\rangle\right\|^2_{2}} = \frac{\|\Phi\|^2_{2,L_p^2}}{\rho^2(\Phi)}
\end{eqnarray*}
where $\rho(\Phi)$ is the spectral radius of the matrix $\langle\Phi,\Phi^T\rangle$.
One finally obtains
\[
\| f^c_N - f_N \|_{L^2_p} \leq  \frac{C}{N^r} \|f\|_{H_p^r}\frac{\|\Phi\|^2_{2,L_p^2}}{\rho(\Phi)}
\]
which proves \eqref{eq:sec}.
\end{proof}

\begin{remark}
The conservative best approximation in least square \eqref{eq:minim1} can be represented in term of the standard projection as
\[
\P_N^c f = \P_N f + \sum_{k=-N}^N \hat C^T_k \langle f-\P_N f,\Phi \rangle.
\]
The above representation emphasizes the analogies with the $L^2$ minimization problem solved in \cite{Gamba:2009}. The main difference is represented by the continuous representation of the solution in the space of trigonometric polynomials which permits to demonstrate spectral accuracy of the resulting approximation. Note also that the same conservative projection remains valid in the case we are interested in performing mesh changes for example by reducing or increasing the number of modes by keeping moment conservation and spectral accuracy. 
\end{remark}


\section{Application to the Boltzmann equation}
  
  The moment-preserving Fourier approximation introduced in Section \ref{sec:specContraint} allows a direct construction of a moment-preserving Fourier-Galerkin approximation of the Boltzmann equation. The algorithmic details of this construction, together with the fast implementation strategy, are reported in Appendix \ref{sec:Boltz}. Here, we will focus our attention on the general mathematical formulation and on the study of its spectral accuracy and stability properties. 
  Let us define the moment constrained Fourier approximation of the truncated Boltzmann operator \eqref{eq:truncatedQ} as the solution of the following problem
  \be
\label{def:constrainedApproxColl}
Q_N^{R,c} (f,f) = {\rm argmin} \left\{\| g_N - Q^{R} (f,f) \|^2_{L^2_p}\,\, :\,\,g_N\in\PP^N,\,\, \langle g_N,\Phi \rangle  = 0\right\}
\ee 
or equivalently
\be 
Q_N^{R,c} (f,f) = {\mathcal P}_N Q^R(f,f) - \sum_{k=-N}^N \hat C_k^T \langle {\mathcal P}_N Q^R(f,f), \Phi\rangle
\label{eq:consQ}
\ee
and the constrained, Fourier projected, homogeneous Boltzman equation then reads
	\be
\left \{
\begin{aligned}
\frac{\partial f_N^c}{\partial t } &= Q_N^{R,c}(f_N^c,f_N^c)\\
f_N^c(v,0)&=\mathcal P_N^c f_0(v),\quad v\in [-\pi,\pi]^d.
\label{eq:homoPNc}
\end{aligned} \right.
\ee
Let us underline that \eqref{eq:consQ} differs from the conservative projection in Definition \ref{def:1} in the sense that the constrained minimization problem \eqref{def:constrainedApproxColl} is solved with respect to the physical conservation laws of the collision term in the whole space and not in the periodic box (see the discussion at the end of Section 2). A direct application of Definition \ref{def:1}, for which we have proved the spectral accuracy property in Theorem \ref{thm:SpectralAccuracyfNc}, will lead to the projected operator
\be
\tilde Q_N^{R,c} (f,f) = {\mathcal P}_N Q^R(f,f) - \sum_{k=-N}^N \hat C_k^T \langle Q^R(f,f) - {\mathcal P}_N Q^R(f,f), \Phi\rangle.
\label{eq:consQt}
\ee 
It is therefore clear, that some smallness on $\langle Q^R(f,f),\Phi\rangle$ is necessary in order to prove spectral consistency of \eqref{eq:homoPNc}. As discussed in Section \ref{ssec:bd}, see Proposition \ref{thm:suppQR}, this is guaranteed if the solution satisfy a smallness assumption outside the ball $\Ball_0(R)$.

In the sequel, we discuss consistency and stability of the moment preserving spectral method.

	\subsection{Spectral consistency}
	First, we shall prove that the moment preserving method for the Boltzmann equation is spectrally accurate.
	Let us recall some of the regularity properties of the truncated collision operator \cite{FM11, PR00}.
	\begin{lemma}[Lemma 4.1 of \cite{FM11}, Lemma 5.2 of \cite{PR00}]
		\label{lem:regQ}
		For all $p \in (1, +\infty]$ there exists a positive constant $C = C(R,p)$ such that if $f \in L^1([-\pi,\pi]^d)$, $g \in L^p([-\pi,\pi]^d)$ one has 
		\begin{equation*}
			\|Q^R(f,g)\|_{L^p} \leq C \|f\|_{L^1} \|g\|_{L^p}.
		\end{equation*}
		Moreover, if $f, g\in L^2([-\pi,\pi]^d)$, one also has
		\begin{equation*}
			\|Q^R(f,g)\|_{L^2} \leq C \|f\|_{L^2} \|g\|_{L^2}.
		\end{equation*}
	\end{lemma}
	
	\begin{proof}
		The proof of these inequalities are classical since the collision kernel is bounded and the functions compactly supported, see e.g. \cite{mouhot2004regularity}.
	\end{proof}  
	
  	Under a smallness assumption on the moments of the truncated collision term (see Proposition \ref{thm:suppQR}) one can then prove a consistency property for equation \eqref{eq:homoPNc}.
  	\begin{theorem}
  		\label{thm:SpectConstQNc}
  		Let $f \in L^2([-\pi,\pi]^d)$ be such that for $\delta > 0$ there exists $R=R(\delta)$ providing the following smallness estimate on the moments of the truncated collision term
		\be
		\|\langle Q^R(f,f),\Phi\rangle\|_2 \leq \tilde C \delta.
		\label{eq:assumpt}
		\ee
		Then for any $r \geq 1$, there exists a constant $C=C\left ( \|f\| _{L^2},R, r\right)$ such that 
  		\begin{equation*}
  			\|Q^R(f,f) - Q_N^{R,c}(f_N^c,f_N^c)\|_{L^2} \leq C \left(\| f - f_N^c \|_{L^2} + \frac{\|Q^R(f_N^c,f_N^c)\|_{H^r_p}}{N^r} + \delta \right).
  		\end{equation*}
  	\end{theorem}
  	
  	\begin{proof}
  		We first split
  		\begin{equation}
		\begin{split}
  			\label{eq:SpectConstSplit}
  			\|Q^R(f,f) - Q_N^{R,c}(f_N^c,f_N^c)\|_{L^2} & \leq  
			\|Q^R(f,f) - Q^{R}(f^c_N,f^c_N)\|_{L^2}\\ &+\|Q^R(f^c_N,f^c_N) - \tilde Q_N^{R,c}(f_N^c,f_N^c)\|_{L^2}
			\\ &+\|\tilde Q_N^{R,c}(f_N^c,f_N^c) - Q_N^{R,c}(f_N^c,f_N^c)\|_{L^2},
			\end{split}
  		\end{equation}
where the projected collision term $\tilde Q_N^{R,c}(f_N^c,f_N^c)$ is defined by \eqref{eq:consQt}.
 		
  		Since $Q^R(f^c_N,f^c_N) \in \mathbb{P}^N \subset H_p^r$, the second term in the RHS of \eqref{eq:SpectConstSplit} is spectrally small, according to Theorem \ref{thm:SpectralAccuracyfNc} there exists $C_1 > 0$ such that
  		\begin{equation}
  			\label{eq:RHS2}
  			\|Q^R(f^c_N,f^c_N) - \tilde Q_N^{R,c}(f^c_N,f^c_N)\|_{L^2} \leq C_1 \frac{\|Q^R(f_N^c,f_N^c)\|_{H^r_p}}{N^r}.
  		\end{equation}
  		Moreover, using Lemma \ref{lem:regQ} and the bilinearity of $Q^R$, there exists $C_2 > 0$ such that
  		\begin{equation}
		\begin{split}
		\label{eq:RHS1}
  			\|Q^R(f,f) - Q^{R}(f^c_N,f^c_N)\|_{L^2} & =  \|Q^{R}(f+f^c_N,f-f^c_N)\|_{L^2}\\
  				& \leq C_2 \, \|f+f_N^c\|_{L^2} \|f-f_N^c|_{L^2}\\
  				& \leq 2 C_2 \|f\| _{L^2} \|f-f_N^c\|_{L^2} 
  		\end{split}
		\end{equation}
		Finally using assumption \eqref{eq:assumpt} there exists $C_3 > 0$ such that
		\be
		\begin{split}
		\|&\tilde Q_N^{R,c}(f_N^c,f_N^c) - Q_N^{R,c}(f_N^c,f_N^c)\|^2_{L^2}  = (2\pi)^d \sum_{k=-N}^N |\hat C^T_k \langle Q^R(f^c_N,f_N^c),\Phi\rangle|^2\\
		& \leq (2\pi)^d  \left(\|\langle Q^R(f,f),\Phi\rangle\|_{2}^2 + \|\langle Q^R(f,f) - Q^{R}(f^c_N,f^c_N),\Phi\rangle\|^2_{2}\right) \sum_{k=-N}^N \| C^T_k\|_2^2\\
		& \leq C_3 (\delta^2+\|f-f_N^c\|^2_{L^2}).
		\end{split}
		\label{eq:RHS3}
		\ee
  		Collecting together \eqref{eq:RHS2}, \eqref{eq:RHS1} and \eqref{eq:RHS3} in \eqref{eq:SpectConstSplit} completes the proof.
  	\end{proof}
  	
	Theorem \ref{thm:SpectConstQNc} states that the rate of convergence of the moment constrained spectral approximation of the truncated Boltzmann collision operator depends on the regularity of the distribution $f$. However, in contrast to the classical spectral estimates in \cite{PR00, MP06}, here the regularity has to be complemented with assumption \eqref{eq:assumpt} on the smallness of the moments of the truncated collision term evaluated at $f$. From a practical viewpoint, this is equivalent to assume a suitable decay of the tails of the initial data and a computational domain large enough to guarantee minimal loss of the collision invariants.
		Gathering this result with the spectral accuracy \eqref{eq:sec} of the moment constrained projection, one then has the following spectral consistency result:
	
	\begin{corollary}
		\label{cor:SpectConstQNc}
		Let $f \in H_p^r([-\pi,\pi]^d)$ for a given $r \geq 1$ be such that there exist $R=R(r,N)$ providing estimate \eqref{eq:assumpt} with $\delta=\delta' N^{-r}$. There exists a constant $C=C\left ( \|f\| _{H^r_p},R\right)$ such that 
  		\begin{equation*}
  			\|Q^R(f,f) - Q_N^{R,c}(f_N^c,f_N^c)\|_{L^2} \leq \frac{C}{N^r} \left(\| f\|_{H_r^p} + \|Q^R(f_N^c,f_N^c)\|_{H^r_p} + \delta'\right).
  		\end{equation*}		
	\end{corollary}

Note that, achieving consistency and spectral accuracy for increasing values of $N$ requires a vanishing error in terms of moments of the collision operator and, as a consequence, a truncation domain which depends on the number of modes $N$. This agrees well with the intuition that a larger computational support has to be used when the number of modes is increased as already observed in practice in \cite{PR00}. 
 
	\subsection{Stability of the moment constrained spectral methods}  	
\label{sec:stabFiMo}
Let us recall the general stability result of Filbet and Mouhot \cite{FM11}, that can be used to prove the stability and large time behavior of spectral methods for the Boltzmann equation. Let us introduced the following perturbed, truncated Boltzmann equation

\begin{equation}
	\left \{\begin{aligned}
	\frac{\partial f}{\partial t } &= Q^R(f,f) + P_\varepsilon(f)\\
f(v,0)&=f_{0,\varepsilon}(v),\quad v\in [-\pi,\pi]^d,
\label{eq:homoPerturbed}
\end{aligned} \right.
\end{equation}
where the perturbation $P_\ve(f)$ is smooth and ``balanced'', in the following sense 
\begin{definition}
	\label{def:stabPert}
	A family of operators $P_\ve$ is said to be a \emph{stable} perturbation of the Boltzmann equation if it verifies the following properties:
	\begin{enumerate}
		\item[i)] $\int P_\ve(f) \,dv = 0$; 
		\item[ii)] there exists  $r \geq 1$, $C_1$, $C_r \geq 0$ such that
		\begin{equation*}
			\left \{ \begin{aligned}
			& \| P_\ve(f) \|_{L^1} \leq C_1 \| f \|_{L^1} \| f \|_{L^1}, \\
			& \| P_\ve(f) \|_{H_p^r} \leq C_r \| f \|_{L^1} \| f \|_{H_p^r}.
			\end{aligned} \right.
		\end{equation*}
		\item[iii)] there exists a function $\varphi(\ve)$ which goes to $0$ when $\ve$ goes to $0$ and such that
		\[
			\| P_\ve(f) \|_{H_p^r} \leq \varphi(\ve), \quad \forall r \geq 1.
		\]
	\end{enumerate}
\end{definition}

One then has the following stability and trends to equilibrium result for the perturbed equation \eqref{eq:homoPerturbed}.
\begin{theorem}[Thm 3.1 of \cite{FM11}]
	\label{thm:FiMoStab}
	Let us consider the perturbed truncated Boltzmann equation \eqref{eq:homoPerturbed}, with a stable family of perturbations $(P_\ve)$ satisfying the hypotheses of Definition \ref{def:stabPert}.
	Let $f_0 \in H^r_p$ for $r > d/2$ be a nonzero, nonnegative   function and $(f_{0,\ve})$ be a family of smooth perturbations of $f_0$:
	\begin{equation*}
		\int_{[-\pi,\pi]^d} f_{0,\ve} \, dv = \int_{[-\pi,\pi]^d} f_{0} \, dv, \qquad \| f_0 - f_{0,\ve}\| \leq \psi(\ve),
	\end{equation*}
	where $\psi(\ve)$ goes to $0$ when $\ve$ goes to $0$. 
	Then there exists $\ve_0$ depending only on the truncation parameter $R$, the collision kernel $B$, the constant in Definition \ref{def:stabPert} and $\|f_0\|_{ H^r_p}$ such that, for any $\ve \in (0,\ve_0)$,
	\begin{enumerate}
		\item there exists a unique global smooth solution $f_\ve$ to \eqref{eq:homoPerturbed};
		\item for any $k < r$, $f_\ve(t, \cdot) \in H_p^k$, uniformly in time;
		\item the quantity of negative values of $f_\ve$ vanishes when $\ve$ goes to $0$;
		\item for any $T >0$, the solution $f_\ve$ of \eqref{eq:homoPerturbed} converges in $L^\infty([0,T];H_p^r)$ towards a solution $f$ to the unperturbed equation \eqref{eq:homo} when $\ve$ goes to $0$;
		\item as time goes to infinity, the solution $f_\ve$ converges in $H_p^r$ towards the piecewise constant equilibrium $m_\infty$ defined in \eqref{def:constEquilibrium}.
	\end{enumerate}	 
\end{theorem}
Note that there are other more recent stability results, see \cite{alonso2018convergence,hu2020new}, but they do not cover the very critical point of the large-time behavior of the methods.

	After investigating the consistency of the moment preserving spectral methods, it is natural to wonder whether the constrained projection impacts the stability and long time behavior properties of the new spectral approach.    
	We shall prove the following theorem.
	\begin{theorem}
		\label{thm:stabiliteConstrainedProjectionBoltzmann}
		Let us consider a nonnegative initial condition $f_0 \in H^k_p([-\pi,\pi]^d)$ for a given $k \geq d/2$ be such that there exist $R=R(k,N)$ providing an uniform in time estimate \eqref{eq:assumpt} for the solution $f(t,v)$ to problem \eqref{eq:homo} with $\delta=\delta' N^{-k+r}$ for any $r < k$.
		There exists $N_0 \in \mathbb N$ depending on the $H_p^k$ norm of $f_0$ and on the spectral radius of the matrix $\langle\Phi,\Phi^T\rangle$ such that for all $N \geq N_0$:
		\begin{enumerate}
			\item There is a unique global smooth solution $f_N^c$ to the problem \eqref{eq:homoPNc};
			\item For any $r < k$, there exists $C_r >0$ such that 
			\[ \|f_N^c(t,\cdot)\|_{H^r_p} \leq C_r; \]
			\item this solution in everywhere positive for time large enough;
			\item this solution $f_N^c(t,\cdot)$ converges to a solution $f(t)$ of the truncated Boltzmann equation \eqref{eq:homo} with spectral accuracy, uniformly in time;
		\end{enumerate}
	\end{theorem}
	
	In order to prove this result, we follow the perturbative framework developed in \cite{FM11} and summarized in Theorem \ref{thm:FiMoStab}. The main difference here is that the constrained method preserves not only mass, but also momentum and kinetic energy, on the finite hypercube $[-\pi,\pi]^d$, without an H-theorem-like decay of the Boltzmann entropy a priori. As such, the equilibria of this new operator are not necessarily Gaussian (not even explicit), and one won't be able to perform the same spectral analysis of the linearized collision operator as was done in \cite{FM11} to study its long time behavior. The same difficulties were faced in \cite{PareschiRey2020} for the equilibrium preserving spectral method. In addition, to recover spectral accuracy we need a smallness assumption on the error in the moments approximation of the collision term.
	
	To mimic the proof of the first points of the general stability Theorem \ref{thm:FiMoStab}, one needs to rewrite the moment constrained spectral method as a stable perturbation of the truncated Boltzmann equation that fits the hypotheses of Definition \ref{def:stabPert}.
	Let us introduce the perturbation operator $P_N^{R,c}$ as
	\begin{equation*}
		P_N^{R,c} (f_N^c) := Q_N^{R,c}(f_N^{c},f_N^{c}) - Q^R(f_N^{c},f_N^{c}).
	\end{equation*}
	Plugging this expression into the constrained, Fourier projected, homogeneous Boltzmann equation \eqref{eq:homoPNc} yields the following perturbed equation
	\begin{equation*}
		\left \{
		\begin{aligned}
		\frac{\partial f_N^c}{\partial t } &= Q^{R}(f_N^c,f_N^c) + P_N^{R,c} (f_N^c)\\
		f_N^c(v,0)&=\mathcal P_N^c f_0(v),\quad v\in [-\pi,\pi]^d.
		\end{aligned} \right.
	\end{equation*}
	
	\begin{proof}[Proof of Theorem \ref{thm:stabiliteConstrainedProjectionBoltzmann}]
		In order to prove the result, one need to check that the perturbation operator $P_N^{R,c}$ is a stable perturbation as in  \ref{def:stabPert}:
		\begin{enumerate}
			\item[i)] the mass conservation is clear, since the constrained projection is conservative
			\[ \int P_N^{R,c}(f)(v) \,dv =  \int \left [ Q_N^{R,c}\left (f,f\right ) - Q^R\left (f,f\right ) \right ] dv = 0.\]
			\item[ii)] Let $r \geq 1$, one has using iteratively Leibniz formula for the derivatives of a product of smooth function on the results of Lemma \ref{lem:regQ}, along with the smoothness of the spectral projection $\mathcal P_N^c$ that
			\begin{align*}
				\left \| P_N^{R,c}(f)\right \|_{H_p^r} & \leq \left \| Q^{R}(f,f) \right \|_{H_p^r} + \left \| Q_N^{R,c}\left (f, f\right ) \right \|_{H_p^{r}} \\
				 & \leq  C_r(R) \, \| f \|_{L^1} \, \|f \|_{H_p^{r}}.
			\end{align*}
			Note that this estimate is similar to the one obtained in \cite{FM11} for the nonconservative spectral method.
			\item[iii)] Using the same regularity estimate of the truncated collision operator, along with the spectral convergence property \eqref{eq:sec} of the constrained spectral projection, one has the spectral smallness of the perturbation
			\begin{align*}
				\| P_N^{R,c}(f)\|_{H_p^r}  & = \left \| Q^{R}(f,f)- Q_N^{R,c}\left (f, f\right ) \right \|_{H_p^r} \\
				 & \leq  \frac{C_{\Phi}}{N^{k-r}} \left(\left \| Q^{R}(f,f) \right \|_{H_p^{r}} + \delta'\right)\\
				 & \leq C_r(R) \frac{C_{\Phi}}{N^{k-r}}  \left(\| f \|_{L^1} \, \|f \|_{H_p^{r}}+\delta'\right).
			\end{align*}
		\end{enumerate}		
		These three properties allow us to obtain the global existence, positivity for large time and spectral convergence (points 1, 3, and 4 of Theorem \ref{thm:stabiliteConstrainedProjectionBoltzmann}) by noticing that for any $f_0 \in H^k_p([\pi,\pi]^d)$ with $k > d/2$ we have
		\[ \| f_N^c(0,\cdot) \|_{H_p^k} \leq \| f_0 \|_{H_p^k}, \qquad \| f_N^c(0,\cdot) - f_0 \|_{H_p^k} \leq \frac{C_\Phi}{N^k}. \]
		
		Finally, the global boundedness of the $H^r$ norm is just a consequence of a classical Grönwall-type argument, that can be reproduced as in \cite{PareschiRey2020}.
		This concludes the proof of Theorem \ref{thm:stabiliteConstrainedProjectionBoltzmann}.
	\end{proof}
      
\section{Numerical examples}
In this section we present several numerical examples to validate our theoretical findings. First we consider the 
moment preserving approximation in the Fourier space and analyze its spectral convergence properties for smooth solutions. Next we compare the results obtained with the new method with those computed using the fast spectral method \cite{MP06}, the equilibrium preserving spectral method recently introduced in \cite{FilbetPareschiRey:2015} and a novel moment preserving and equilibrium preserving spectral method obtained by combining the two previous approaches.

\subsection{Conservative approximations in the Fourier space}
	\label{sec:ConsApprox1D}
	We are interested in validating the theoretical results on the moment constrained spectral representation \eqref{eq:cp}-\eqref{eq:minim1} for the first three moments, namely mass, momentum and kinetic energy.  
\paragraph*{Test 1. Conservation of moments and spectral accuracy}
	We consider the following simple problem to verify conservations and spectral accuracy:
		Compute the unidimensional truncated spectral projection \eqref{eq:truncatedFourierSum} $f_N$ of a given function $f$ with moments $U$, then compute its moment constrained Fourier transform $f_N^c$ using formula \eqref{eq:cp}-\eqref{eq:minim1} by prescribing the moments of the original function $f$ and compute both the error on the moments $U_N^c := (\rho_N^c, \rho u_N^c, \rho e_N^c)$ of $f_N^c$ with respect to $U$, and the $L^2$-error between the original $f$ and $f_N^c$.
		
We shall perform this numerical test for the reduced centered Gaussian
		\begin{equation}
			\label{eq:RCgaussian}
			f(v) = \frac 1{\sqrt{2 \pi}} \exp\left (-v^2/2 \right ), \quad v \in [-6,6],
		\end{equation}
		and for the following asymmetric sum of Gaussians
		\begin{equation}
			\label{eq:twoBumps}
			f(v) = \frac 1{2\sqrt{2 \pi}} \left (\exp\left (-(v-4)^2/2 \right ) + \exp\left (-(v+2)^2/2 \right ) \right ), \quad v \in [-12,12].
		\end{equation}
		
		\begin{table}
			\begin{center}
			{
			\begin{tabular}{c|cccc}
			\multicolumn{5}{c}{function \eqref{eq:RCgaussian}}\\
				\hline\\[-.2cm] 				
				N & $\left |\rho - \rho_N^c\right |$ & $\left |\rho u - \rho u_N^c\right |$ & $\left |\rho e - \rho e_N^c\right |$ & $\left \|f - f_N^c\right \|_{{2}}$	  \\[+.1cm]
				\hline {8} & 0 & {8.462e-16} & {2.776e-16} & {1.654e-09} \\	
				\hline {16} & 0 & {6.540e-17} & {3.053e-16} & {1.443e-13}\\			
				\hline {32} & 0 & {1.496e-16} & {5.551e-16} & {3.342e-16} \\			
				\hline\\[-.2cm] 
				\multicolumn{5}{c}{function \eqref{eq:twoBumps}}\\
				\hline\\[-.2cm]
				N & $\left |\rho - \rho_N^c\right |$ & $\left |\rho u - \rho u_N^c\right |$ & $\left |\rho e - \rho e_N^c\right |$ & $\left \|f - f_N^c\right \|_{{2}}$	  \\[+.1cm]
				\hline {8} & 0 & {1.756e-16} & {8.770e-15} & {8.402e-05} \\			
				\hline {16} & 0 & {2.165e-16} & {6.808e-16} & {1.545e-11}\\			
				\hline {32} & 0 & {3.849e-16} & {1.314e-16} & {6.007e-15} \\			
				\hline\\[-.25cm] 	
			\end{tabular}
			\caption{\textbf{Test 1.} Approximation errors of the moment constrained Fourier truncation $f_N^c$ in \eqref{eq:cp}-\eqref{eq:minim1} for the reduced centered Gaussian \eqref{eq:RCgaussian} and for the two bumps function \eqref{eq:twoBumps}.}}
			\end{center}	
			\label{tab:errorGaussian6}		
		\end{table}

		We present in Table \ref{tab:errorGaussian6} the errors on the mass, momentum, energy and on the distribution for the moment constrained spectral projection. As expected, the constrained projection preserves the moments up to machine precision even in the case of an asymmetric distribution. We also observe the spectral convergence of the constrained distribution $f_N^c$ towards the exact solution $f$ as predicted in Theorem \ref{thm:SpectralAccuracyfNc}.
		
		We emphasize that in Table \ref{tab:errorGaussian6} we used the exact Fourier expansion of the moment vector given by \eqref{eq:CFE}. One other possibility would be to use a Discrete Fourier Transform on the original moment vector, by sampling uniformly its value (and using the discrete version of Parseval identity, see identity (6.1) in \cite{BcHR:2019}). This, however, will introduce an additional spectrally small error on the moment of the constrained function $f_N^c$ which may reduce accuracy when using a small number of nodes.
	
	\subsection{Moment constrained and equilibrium preserving spectral methods}
		Let us now apply the conservative approximations to the fast spectral method for the Boltzmann equation \eqref{eq:homo}.
		For the velocity discretization, we choose $\Omega = [-12,12]^2$, $N = 64^2$ then $128^2$ points and $M = 8$ angular discretizations. Because the problem is not stiff, we use an explicit Runge-Kutta of order 4 in time, with $\Delta t = 0.01$. We shall compare the fast spectral (\textbf{FS}) method \cite{MPR:13} with the moment preserving fast spectral (\textbf{MPFS}) method \eqref{eq:ode}. We shall also compare our numerical experiments with  the equilibrium preserving fast spectral (\textbf{EPFS}) method introduced in \cite{FilbetPareschiRey:2015}, and by taking a combination of the two approaches where we apply the equilibrium preserving method together with the moment preserving technique. Note that, this latter approach (referred to as \textbf{MEPFS}) preserves not only the moments but also the local Maxwellian equilibrium state. For the reader convenience, its details are summarized in Appendix \ref{rk:app} (see equation \eqref{eq:mepfs}). 
		
      	In the following, for the sake of brevity, we omit to present the error for short times since all three methods give similar results due to their spectral convergence properties, instead we will focus on the behavior of the methods for long times.
	\smallskip
	
	\paragraph{Test 2. Trends to equilibrium}
	
		We shall consider an exact solution of the homogeneous Boltzmann equation, the so called Bobylev-Krook-Wu solution \cite{Bobylev:75,KrookWu:1977}. 
		It is given in two dimensions of velocity by
		\begin{equation}
			\label{eq:fBKW}
		  	f_{BKW}(t,v) = \frac{\exp(-v^2/2S)}{2\pi S^2} \,\left[2\,S-1+\frac{1-S}{2 \,S}\,v^2 \right]
		\end{equation}
		with $S = S(t) = 1-\exp(-t/8)/2$. 
		 	
      \begin{figure}
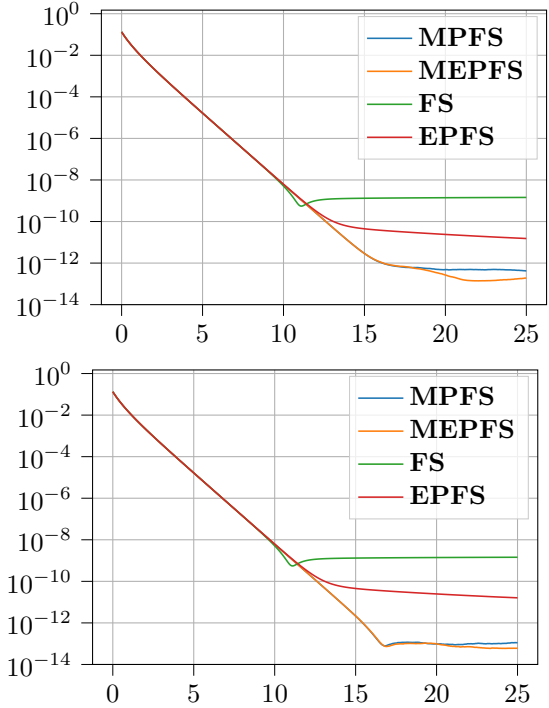

      	\begin{center}
        \input{Pics/ComparisonRelativeEntropy_32.tex}\\
        \input{Pics/ComparisonRelativeEntropy_64.tex}
        \caption{\textbf{Test 2.} Time evolution of $\|f_N - M\|_{L^2}$ for the BKW solution \eqref{eq:fBKW}, where $M$ is the local Maxwellian, using $N=32^2$ (top) and $N=64^2$ (bottom).}
        \end{center}
        \label{fig:BoltzmannRelEnt}
      \end{figure}        
 
      Figure \ref{fig:BoltzmannRelEnt} presents the time evolution of the $L^2$ error of the solution $f_N$ with respect to the equilibrium distribution $M$, where we observe an exponential convergence towards $M$.      
      As observed in \cite{FilbetPareschiRey:2015,PareschiRey:2016}, the behavior of the \textbf{EPFS} method is better than the classical \textbf{FS} which saturates around $10^{-9}$, but both are outperformed by almost two order of magnitude by the new moment constrained \textbf{MPFS} method. Adding the equilibrium preserving feature to this latter method slightly improve its accuracy, and the \textbf{MEPFS} scheme outperforms all the others in very large time. 

\smallskip
      
      \paragraph{Test 3. Error on the temperature}
      
      We now consider the following asymmetric two bumps initial data:
      \begin{equation}
      	\label{eq:twoBumps2D}
      	f(v) = \frac 1{4 \pi} \left (e^{-(|v-1|^2 + |v-2|^2)/2} + e^{-(|v+2|^2 + |v+1|^2)/2}\right ), \quad \forall\,\, v \in [-12,12]^2.
      \end{equation}

      \begin{figure}
        \begin{center}
        \includegraphics[scale=0.8]{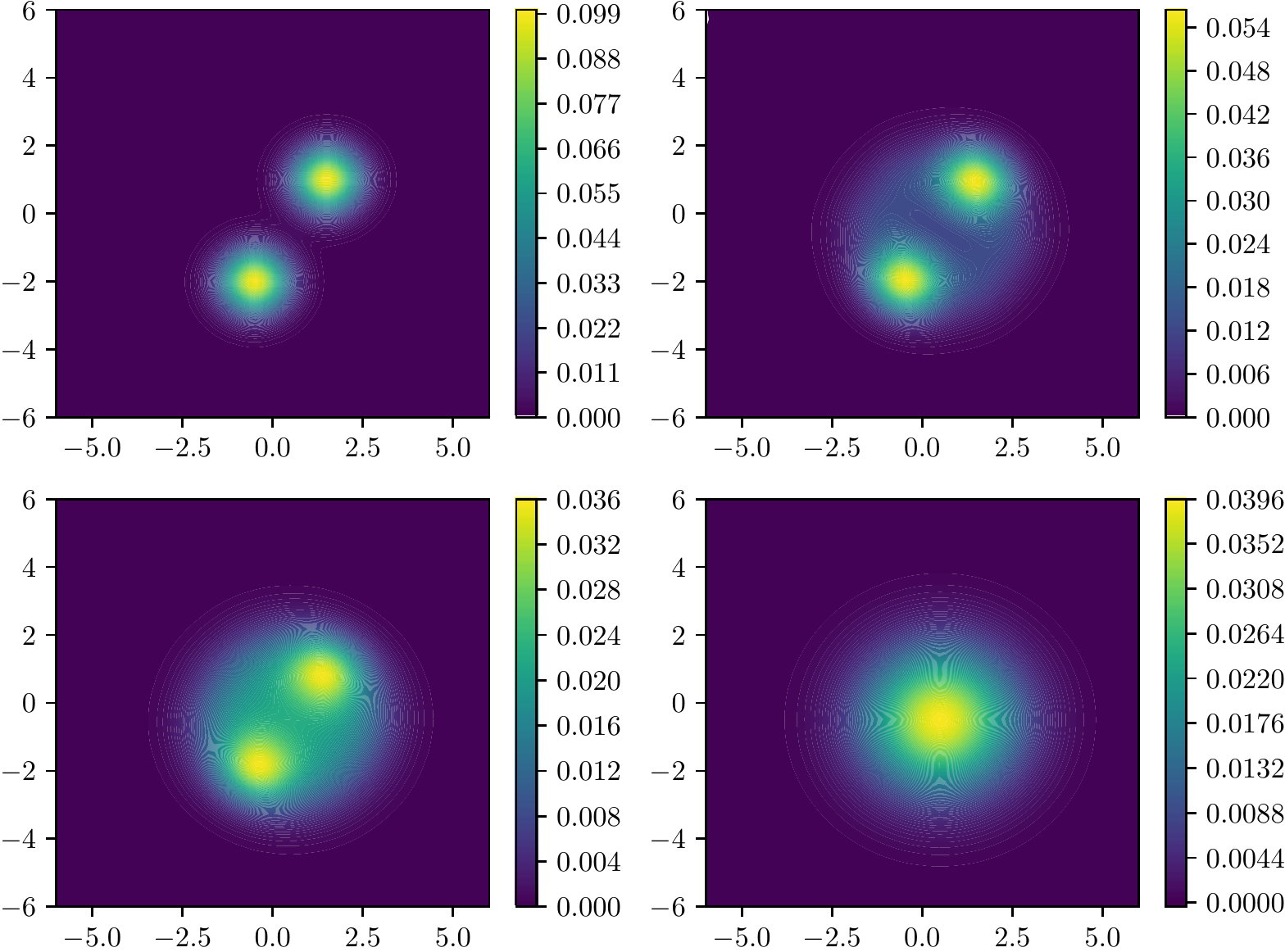}
        \caption{\textbf{Test 3.} Contour plot of the solution $f_N^c$ of the \textbf{MPFS} method, for the two bumps initial data \eqref{eq:twoBumps2D}, with $N=64^2$ points, at time $t=0, \ 0.5, \ 1, \ 10$ (bottom).}
        \label{fig:Boltzmann_EvoContour}
        \end{center}
      \end{figure} 
           
      Figure \ref{fig:Boltzmann_EvoContour} presents the isovalues of the solution computed with the \textbf{MPFS} method, for $N = 64^2$. We can observe the correct convergence towards the Maxwellian equilibrium state. 

      \begin{figure}
        \begin{center}
        \includegraphics[scale=0.7]{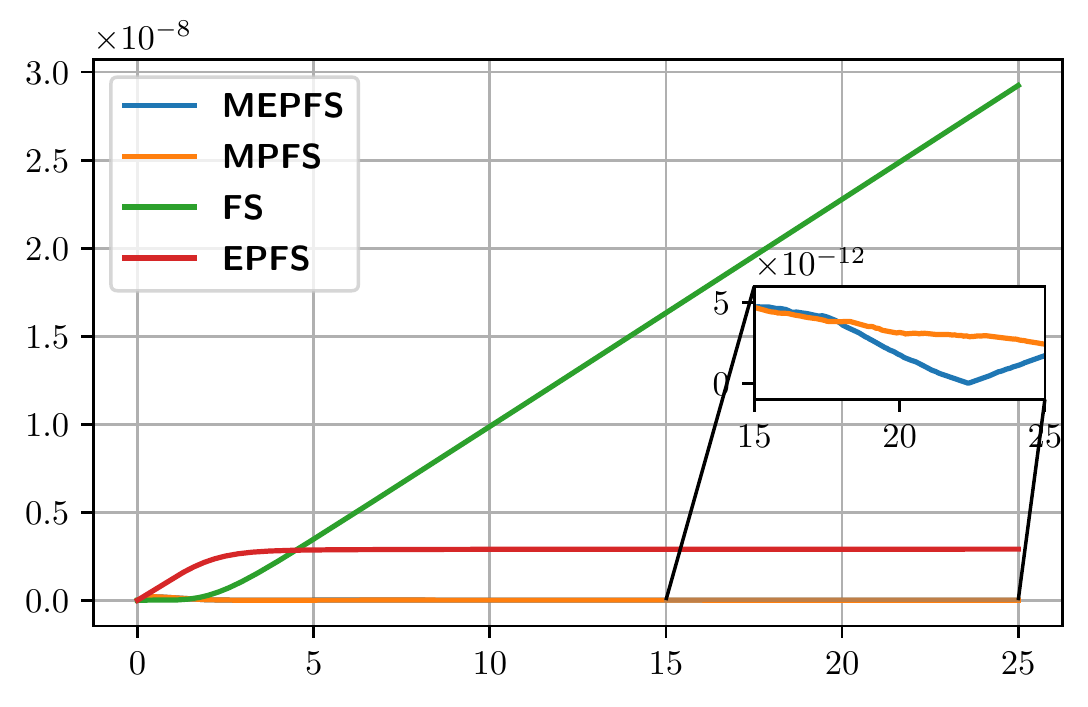}
        \includegraphics[scale=0.7]{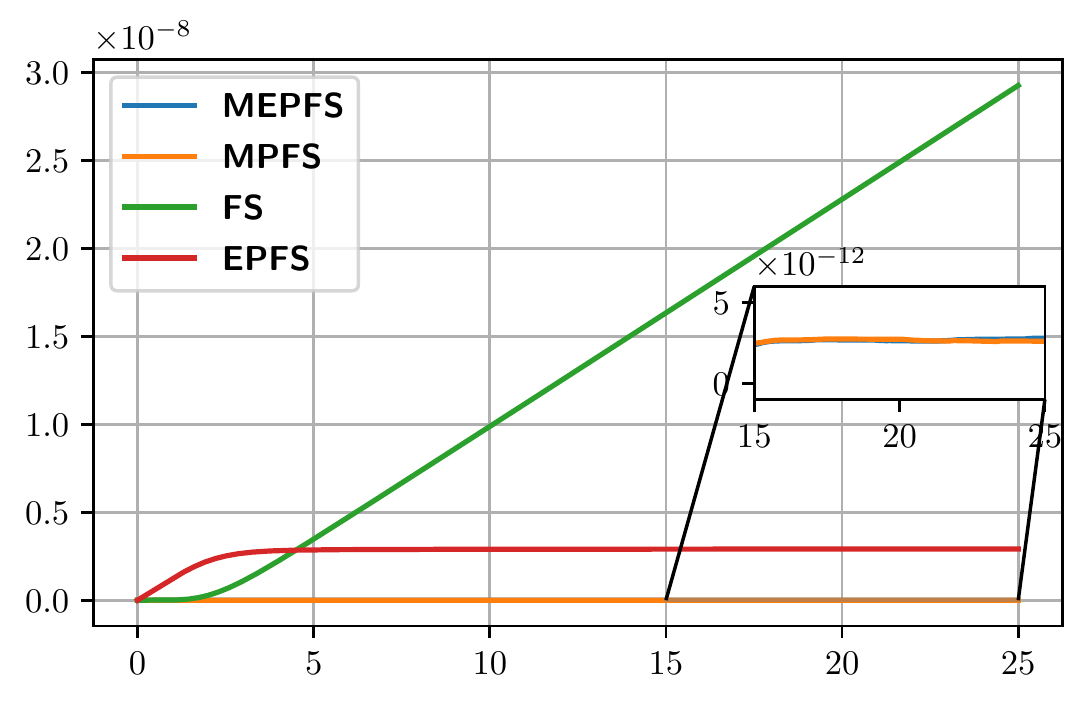}
        \caption{\textbf{Test 3.} Error on the temperature for the two bumps initial data \eqref{eq:twoBumps2D}, for $N=32^2$ (top) and $N=64^2$ (bottom).}
        \label{fig:Boltzmann_Temperature}
        \end{center}
      \end{figure} 

      We then compute numerically the time evolution of the temperature \[ T=\frac1{2\rho}\int f |v-u|^2 \, dv \]
      associated to this test case.
      Figure \ref{fig:Boltzmann_Temperature} shows the temperature error $|T_N(t) - T(0)|$ of the three numerical methods.
      As expected, for the classical spectral method \textbf{FS} one can observe a linear growth of the error. This growth is almost grid independent, and is actually due to the Fourier truncation. As was observed in \cite{FilbetPareschiRey:2015}, using the equilibrium preserving correction \textbf{EPFS} improves this conservation error by one order of magnitude, and again independently on the grid.
      Nevertheless, one needs to perform the moment constrained correction \textbf{MPFS} in order to (almost) attain the machine precision in the evolution of the temperature. 
      Note that including the equilibrium preserving approach \textbf{MEPFS} improves the results only marginally.
      

\section{Conclusions}
We introduced and analyzed a new class of Fourier-Galerkin spectral methods for kinetic equations that can preserve the moments of the distribution function. The method was introduced using a best constrained approximation formalism in the space of trigonometric polynomials. Due to the general formulation, the method allows the spectral accuracy property of the solution to be maintained for the conservative finite Fourier series. The new approximation was then used to derive fast Fourier-Galerkin methods for the Boltzmann equation that preserve the collisional invariants, and their theoretical properties have been analyzed. Compared to other conservative schemes for the Boltzmann equation based on a constrained minimization approach \cite{Gamba:2009}, the analysis of the theoretical properties of the new method, such as spectral consistency, is greatly simplified thanks to the Fourier-Galerkin setting. Due to the enforcement of conservations, the corresponding estimates contain an additional error term that depends on the smallness of the moments of the collision operator. We also introduced a modification of the method which is capable to preserve exactly the equilibrium state represented by the conservative projection of the local Maxwellian.

Given its generality, the method admits numerous extensions. Among the most interesting are certainly the construction of spectrally accurate and conservative methods for the Landau equation. Another interesting direction is the application to kinetic models in the socio-economic domain where equilibrium states are often unknown and thanks to the present approach can be computed with spectral accuracy.  

\medskip
\paragraph{Acknowledgment}
This work has been written within the
activities of GNCS groups of INdAM (National Institute of
High Mathematics). L.P. acknowledge the partial support of MIUR-PRIN Project 2017, No. 2017KKJP4X “Innovative numerical methods for evolutionary partial differential equations and applications”. T.R. would like to thanks Maxime Herda for fruitful discussions on the implementation of the method. T.R. was partially funded by Labex CEMPI (ANR-11-LABX-0007-01) and ANR Project MoHyCon (ANR-17-CE40-0027-01).

\appendix

\section{The conservative fast spectral method} \label{sec:Boltz}
The fast spectral method developed in \cite{MP06} can be applied directly in the conservative formulation. In the sequel we will summarize briefly the details of the method. 
Since the collision operator is local in space and time, only the dependency on the velocity variable $v$ is considered for the distribution function $f$, i.e. $f=f(v)$. 
We suppose the distribution function $f$ to have compact support on the ball $\Ball_0(R)$ of
radius $R$ centered in the origin. Then, since one can prove \cite{PR00} that 
\[\supp (Q(f)(v)) \subset \Ball_0({\sqrt 2}R).\] 
In order to write a spectral approximation which avoid aliasing, it is sufficient that the distribution function $f(v)$ is restricted on the cube $[-T,T]^{d}$ with $T \geq (2+{\sqrt 2})R$. Successively, one should assume $f(v)=0$ on $[-T,T]^{d} \setminus \Ball_0(R)$ and extend $f(v)$ to a periodic function on the set $[-T,T]^{d}$. Let observe that the lower bound for $T$ can be improved. For instance, the choice $T=(3+{\sqrt 2})R/2$ guarantees the absence of intersection between periods where $f$ is different from zero. However, since in practice the support of $f$ increases with time, we can just minimize the errors due to aliasing \cite{canuto:88} with spectral accuracy. 

To further simplify the notation, let us take $T=\pi$ and hence $R=\lambda\pi$ with $\lambda = 2/(3+\sqrt{2})$ in the following. 
Hereafter, using one index to denote the $d$-dimensional sums, we have that the conservative approximate Fourier truncation $f^c_N$ can be represented as 
\begin{equation*}
f^c_N(v) = \sum_{k=-N}^{N} \f^c_k e^{i k \cdot v},\qquad \f^c_k = \f_k + \hat C_k^T\left(\langle f,\Phi \rangle-\langle f_N,\Phi \rangle\right)
\end{equation*}
where $\f_k$ are the standard Fourier coefficients in \eqref{eq:FS}, $\Phi$ is the moment vector defined in Proposition \ref{pr:2}, and $\hat C_k^T$ is defined in \eqref{eq:minim1}.

\subsection{A conservative spectral quadrature}
We then obtain a conservative spectral quadrature of our collision operator by a constrained projection of $Q^R$ defined in \eqref{def:constrainedApproxColl} on the space
of trigonometric polynomials of degree less or equal to $N$, i.e.
\begin{equation*}
Q^{R,c}_N(f_N^c,f_N^c)(v) = \sum_{k=-N}^{N} \hat Q^c_k e^{i k \cdot v}
\end{equation*}
where using Definition \ref{def:1} we have
\begin{equation}
{\hat Q}^c_k= \sum_{\substack{l,m=-N\\l+m=k}}^{N} \f^c_l\,\f^c_m \,\hat{\beta}(l,m)
-\hat C_k^T \langle Q_N^R(f^c_N, f^c_N),\Phi \rangle
, \quad k=-N,\ldots,N. 
\label{eq:CF1}
\end{equation}
In the above equation $\hat{\beta}(l,m)=\B(l,m)-\B(m,m)$ are given by
\begin{equation*}
\B(l,m) = \int_{\Ball_0(2\lambda\pi)}\int_{\ens{S}^{d-1}} 
|q| \sigma(|q|, \cos\theta) e^{-i(l\cdot q^++m\cdot q^-)}\,d\omega\,dq 
\end{equation*}
with \begin{equation*}
q^{+} = \frac12(q+\vert q\vert \omega), \quad
q^{-} = \frac12(q-\vert q\vert \omega).
\end{equation*}
Let us notice that the naive evaluation of (\ref{eq:CF1}) requires $O(n^2)$ operations, where $n=N^d$. This causes the spectral method to be computationally very expensive,
especially in dimension three. In order to reduce the number of operations needed to evaluate the collision integral, the main idea is to use another representation of the collision operator, the so-called Carleman representation \cite{Carl:EB:32} which is obtained by using the  following identity
  \begin{equation*}
    \frac{1}{2} \, \int_{\mathbb{S}^{d-1}} F(|u|\sigma - u) \, d\sigma	 = \frac{1}{|u|^{d-2}} \, \int_{\RR^{d}} \delta(2 \, x \cdot u + |x|^2) \, F(x) \, dx.
 \end{equation*} 
This gives in our context the Boltzmann integral representation
\begin{equation}
\label{defQBCarleman}
 Q(f,f)= \int_{\R^{d}} \int_{\R^{d}} {\tilde B}(x,y) 
 \delta(x \cdot y) 
 \left[ f(v + y) \, f(v+ x) - f(v+x+y) \, f(v) \right] \, dx \,
 dy,
 \end{equation} 
with 
  \begin{equation}\label{eq:Btilde}
  \tilde{B}(|x|,|y|) =
  2^{d-1} \, \sigma\left(\sqrt{|x|^2+|y|^2}, \frac{|x|}{\sqrt{|x|^2+|y|^2}} \right) \, (|x|^2+|y|^2)^{-\frac{d-2}2}.
  \end{equation}
Denoting by $Q^{F,R}(f,f)$ the truncation on $\Ball_0(R)$ of \eqref{defQBCarleman} we have the following conservative spectral quadrature formula 
 \begin{equation}\label{eq:ode}
 \hat{Q}^{F,c}_k  =
 \sum_{\underset{l+m=k}{l,m=-N}}^{N} {\hat{\beta}}_F(l,m) \, \hat{f}^c_l \, \hat{f}^c_m-\hat C_k^T \langle Q_N^{F,R}(f^c_N, f^c_N),\Phi \rangle, \ \ \
 k=-N,...,N
 \end{equation}
where ${\hat{\beta}}_F(l,m)=\B_F(l,m)-\B_F(m,m)$ are now 
given by
 \begin{equation*}
 \B_F(l,m) = \int_{\Ball_0(R)} \int_{\Ball_0(R)}
 \tilde{B}(x,y) \, \delta(x \cdot y) \, 
 e^{i (l \cdot x+ m \cdot y)} \, dx \, dy.
 \end{equation*}

 \subsection{Other conservative and equilibrium preserving methods}
 \label{rk:app}
	The constrained modes of the collision term in \eqref{eq:CF1} (or the fast ones \eqref{eq:ode}) have the general form
	\begin{equation}
		\hat Q_k^{c} (f^c_N,f^c_N) = \hat Q_k(f^c_N,f^c_N) - \hat C^T_k \langle Q_N^{R}(f^c_N, f^c_N),\Phi \rangle.
	\label{eq:mpfs}
	\end{equation}  
	The above representation is reminiscent of the conservative spectral method introduced in \cite{Gamba:2009}, with  one major difference: the latter method performs the moment constrained projection using the grid points in the original velocity space and a finite difference discretization in the Fourier space. As a consequence its theoretical analysis is more difficult, in particular regarding its consistency properties \cite{alonso2018convergence}. Although for the sake of brevity we have omitted the results in the numerical section, the constrained approach in \cite{Gamba:2009} was also applied in combination with the classical Fourier-Galerkin approximation, and the accuracy obtained was comparable to that of the conservative spectral method, indicating that spectral consistency could also be demonstrated for this latter approach. The same thing was also suggested in \cite{wu2013deterministic} by combining the constrained projection in the original velocity space using the Fourier-Galerkin approximation in \cite{MPR:13}. However, we have not explored these analogies further from a theoretical viewpoint, as they are beyond the scope of this article.\\	
	Let us also note that the equilibrium-preserving fast spectral method in \cite{FilbetPareschiRey:2015,PareschiRey2020} has also a similar form and the corresponding Fourier modes can be written as
	\begin{equation}
		\hat Q_k^{e} (f_N,f_N) = \hat Q_k(f_N,f_N) - \hat Q_k(M_N,M_N),
		\label{eq:epfs}
	\end{equation} 
	where $M_N=\mathcal{P}_N M$ and $M$ is the local Maxwellian equilibrium. The two projections \eqref{eq:mpfs} and \eqref{eq:epfs} can be combined to originate a conservative and equilibrium preserving spectral method defined as
	\begin{equation}
	\begin{split}
		\hat Q_k^{c,e} (f^c_N,f^c_N) &= \hat Q^c_k(f^c_N,f^c_N) - \hat Q^c_k(M^c_N,M^c_N)\\
		& = \hat Q_k(f^c_N,f^c_N) -
		\hat Q_k(M^c_N,M^c_N)\\
		&\quad - \hat C^T_k \left(\langle Q_N^{R}(f^c_N, f^c_N)- Q_N^{R}(M^c_N, M^c_N),\Phi \rangle\right),
		\end{split}
		\label{eq:mepfs}
	\end{equation} 
	where now $M^c_N=\mathcal{P}^c_N M$. The spectral consistency of the method in \eqref{eq:mepfs} is a direct consequence of the theoretical results in this article and we will omit the details. 
	The formulations \eqref{eq:mpfs}, \eqref{eq:epfs} and \eqref{eq:mepfs} show a common analogy between all the different approaches, namely to use a spectrally small correction to recover the conservation properties and/or the correct equilibrium state.

\subsection{The fast algorithm}
To reduce the number of operation needed to evaluate~\eqref{eq:ode}, we look for a convolution structure. 
The aim is to approximate each
${\hat{\beta}}_F(l,m)$ by a sum
 \[ {\hat{\beta}}_F(l,m) \simeq \sum_{p=1} ^{A} \alpha_p (l) \alpha' _p (m), \]
where $A$ represents the number of finite possible directions of collisions.
This finally gives a sum of $A$ discrete convolutions and, consequently, the algorithm can be computed in $O(A \, N \log_2 N)$ operations by means of
standard FFT technique~\cite{canuto:88}. 

In order to get this convolution form, we make the decoupling assumption
 \begin{equation*}
 \tilde{B}(x,y) = a(|x|) \, b(|y|).
 \end{equation*}
This assumption is satisfied if $\tilde B$ is constant. This is the case of Maxwellian molecules in dimension two, and hard spheres in dimension three. 
Indeed, using kernel \eqref{defVHSKernel} in \eqref{eq:Btilde}, one has
\[
  \tilde{B}(x,y) = 2^{d - 1} C_\alpha (|x|^2+|y|^2)^{-\frac{d-\alpha-2}2},
\] 
so that $\tilde{B}$ is constant if $d = 2$, $\alpha = 0$ and $d = 3$, $\alpha = 1$.

\begin{remark}
Let us detail the computations in dimension $2$ for $\tilde{B} =1$, i.e. Maxwellian molecules. Here we write $x$ and $y$ in spherical coordinates $x = \rho e$ and $y = \rho' e'$ to get
 \begin{equation*}
 \B_F(l,m) = \frac14 \, \int_{\ens{S}^1} \int_{\ens{S}^1}
 \delta(e \cdot e') \,
 \left[ \int_{-R} ^R e^{i \rho (l \cdot e)} \, d\rho \right] \,
 \left[ \int_{-R} ^R e^{i \rho' (m \cdot e')} \, d\rho' \right] \, de \,
 de'.
 \end{equation*}
Then, denoting $ \phi_R ^2 (s) = \int_{-R} ^R e^{i \rho s} \, d\rho,$ for $s \in \R$,  
we have the explicit formula
 \[ \phi_R ^2 (s) = 2 \, R \,{\sinc} (R s), \]
 where ${\sinc}(x)=\frac{\sin(x)}{x}$ for $x \neq 0$.
This explicit formula is further plugged in the expression of $\B_F(l,m)$ and using its parity property, this yields
 \begin{equation*}
 \B_F (l,m) =  \int_0 ^{\pi} \phi_R ^2 (l \cdot e_{\theta})\, \phi_R ^2 (m \cdot e_{\theta+\pi/2}) \,
 d\theta.
 \end{equation*}
Finally, a regular discretization of $A$ equally spaced points, which is spectrally accurate because of the periodicity of the function, gives
 \begin{equation*}
 \B_F (l,m) = \frac{\pi}{M} \, \sum_{p=1} ^{A} \alpha_p (l) \alpha' _p (m),
 \end{equation*}
with
 \[
 \alpha _p (l) = \phi_R ^2 (l \cdot e_{\theta_p}), \hspace{0.8cm} \alpha' _p (m) = \phi_R ^2 (m \cdot e_{\theta_p+\pi/2}) 
 \]
where $\theta_p = \pi p/A$.
\end{remark}

In practice, for the two dimensional case in velocity we choose a number of discretization angles $A=8$. 
This is enough to guarantee a good accuracy of the results in many situations. However, when close to shock waves or a boundary layers, it may happen that spectral accuracy is lost \cite{gamba2017fast}. In these situations, in order to keep 
a good accuracy in the resolution of the collision operator, one may need a larger set of angles.

%
%
%
%

\bibliographystyle{acm}
\bibliography{biblioPR}

\end{document}